\documentclass[oneside]{amsart}
\usepackage[stable]{footmisc}
\usepackage{setspace}
\usepackage{sseq,pgffor,comment}
\usepackage[a4paper]{geometry}
\usepackage{amsmath,amsthm,amssymb,array,enumerate,graphicx,etoolbox,tabularx,mathtools}
\usepackage[all]{xy}
\usepackage{tikz-cd}
\usepackage{xcolor} 
\usepackage[pagebackref]{hyperref}
\definecolor{dark-red}{rgb}{0.5,0.15,0.15}
\definecolor{dark-blue}{rgb}{0.15,0.15,0.6}
\definecolor{dark-green}{rgb}{0.15,0.6,0.15}
\hypersetup{
    colorlinks, linkcolor=dark-red,
    citecolor=dark-blue, urlcolor=dark-green
}

\renewcommand*{\backref}[1]{}
\renewcommand*{\backrefalt}[4]{%
  \ifcase #1 %
No citations.
  \or
(cit. on p. #2).%
  \else
(cit on pp. #2).%
  \fi%
}
\numberwithin{equation}{section}
\usepackage[nameinlink,capitalise,noabbrev]{cleveref}
\newcount\tmp
\newcommand{\Z}{\mathbb{Z}}
\newcommand{\Q}{\mathbb{Q}}
\newcommand{\G}{\mathbb{G}}

\newcommand{\W}{\mathbb{W}}
\newcommand{\Zp}{\Z_{(p)}}
\newcommand{\F}{\mathbb{F}}

\newcommand{\xr}{\xrightarrow}

\newcommand{\Hom}{\operatorname{Hom}}
\newcommand{\Comod}{\operatorname{Comod}}
\newcommand{\wComod}{\widehat{\Comod}}
\newcommand{\Tor}{\operatorname{Tor}}
\newcommand{\Ext}{\mathrm{Ext}}

\newcommand{\Mod}{\widehat{\operatorname{Mod}}}
\newcommand{\holim}{\operatorname{holim}}
\newcommand{\wExt}{\widehat{\Ext}}

\newcommand{\btimes}{\boxtimes}

\newcommand{\fm}{\mathfrak{m}}
 
\DeclareMathOperator{\Aut}{Aut}
\DeclareMathOperator{\Gal}{Gal}

\usepackage{cleveref}

\newtheorem{theorem}{Theorem}[section]

\newtheorem{lemma}[theorem]{Lemma} 
\newtheorem{cor}[theorem]{Corollary}
\newtheorem{example}[theorem]{Example}
\newtheorem{prop}[theorem]{Proposition} \theoremstyle{definition}
\newtheorem{rem}[theorem]{Remark} 
\newtheorem{definition}[theorem]{Definition}

\Crefname{cor}{Corollary}{Corollaries}
\Crefname{conjecture}{Conjecture}{Conjectures}
\Crefname{rem}{Remark}{Remarks}
\Crefname{question}{Question}{Questions}
\Crefname{prop}{Proposition}{Propositions}	
\newcommand{\mE}{E^\vee_*}
\newcommand{\htimes}{\widehat{\otimes}}
\AtBeginDocument{%
   \def\MR#1{}
}
\renewcommand{\frak}{\mathfrak}
\begin{document}
\title{The $E_2$-term of the $K(n)$-local $E_n$-Adams spectral sequence}

\author{Tobias Barthel}
\address{Max Planck Institute for Mathematics, Bonn, Germany}
\author{Drew Heard}
\address{Max Planck Institute for Mathematics, Bonn, Germany}
\date{\today}

\begin{abstract}
    Let $E=E_n$ be Morava $E$-theory of height $n$. In~\cite{devhop04} Devinatz and Hopkins introduced the $K(n)$-local $E_n$-Adams spectral sequence and showed that, under certain conditions, the $E_2$-term of this spectral sequence can be identified with continuous group cohomology.  We work with the category of $L$-complete $\mE E$-comodules, and show that in a number of cases the $E_2$-term of the above spectral sequence can be computed by a relative $\Ext$ group in this category. We give suitable conditions for when we can identify this $\Ext$ group with continuous group cohomology. 
\end{abstract}
\maketitle
\section*{Introduction}

Let $E_n$ denote the $n$-th Morava $E$-theory (at a fixed prime $p$), the Landweber exact cohomology theory with coefficient ring
\[
\pi_*(E_n) = \W(\F_{p^n})[\![u_1,\ldots,u_{n-1}]\!][u^{\pm 1}],
\]
where each $u_i$ is in degree 0, and $u$ has degree -2. Here $\W(\F_{p^n})$ refers to the Witt vectors over the finite field $\F_{p^n}$ (an unramified extension of $\Z_p$ of degree $n$). Note that $E_0$ is a complete local regular Noetherian ring with maximal ideal $\frak{m} = (p,u_1,\ldots,u_{n-1})$.  

Unless indicated otherwise, let us fix an integer $n \ge 1$ and write $E$ instead of $E_n$ throughout. The cohomology theory $E$ plays a very important role in the chromatic approach to stable homotopy theory, in particular in the understanding of the $K(n)$-local homotopy category (see, for example,~\cite{hs99}). 

The formal group law associated to $E_*$ is the universal deformation of the Honda formal group law $\Gamma_n$ of height $n$ over $\F_{p^n}$. Let $\G_n = \Aut(\Gamma_n) \rtimes \Gal(\F_{p^n}/\F_p)$ denote the $n$-th (extended) Morava stabilizer group. Lubin-Tate theory implies that $\G_n$ acts on the ring $E_*$, and Brown representability implies that $\G_n$ acts on $E$ itself in the stable homotopy category. The Goerss--Hopkins--Miller theorem~\cite{rezkhm,gh04} implies that this action can be taken to be via $E_\infty$-ring maps. 

In general $\G_n$ is a profinite group, and it is not clear how to form the homotopy fixed points with respect to such groups (although progress has been made in this area; see~\cite{davis06,behdav,quicklt}). Nonetheless, in~\cite{devhop04} Devinatz and Hopkins defined $E_\infty$-ring spectra $E^{hG}$  for $G \subset \G_n$ a closed subgroup of the Morava stabilizer group, which behave like continuous homotopy fixed point spectra (and indeed if $G$ is finite they agree with the usual construction of homotopy fixed points). Remarkably they showed that there is an equivalence $E^{h\G_n} \simeq L_{K(n)}S^0$, a result expected since the work of Morava~\cite{mor85}. Davis~\cite{davis06}, Behrens--Davis~\cite{behdav} and Quick~\cite{quicklt} have given constructions of homotopy fixed point spectra with respect to the continuous action of $G$ on $E$, and these agree with the construction of Devinatz and Hopkins. 

Devinatz and Hopkins additionally showed that for any spectrum $Z$ there is a strongly convergent spectral sequence 
\[
E_2^{\ast,\ast}=H_c^*(G,E^\ast Z) \Rightarrow (E^{hG})^\ast Z
\]
which is a particular case of a spectral sequence known as the $K(n)$-local $E$-Adams spectral sequence~\cite[Appendix A]{devhop04}. Here, for a closed subgroup $G \subset \G_n$ the continuous cohomology of $G$ with coefficients in a topological $\G_n$-module $N$ is defined using the cochain complex $\Hom^c(G^\bullet,N)$ (see the discussion before the proof of~\Cref{thm:COR}).

 Using homology instead of cohomology Devinatz and Hopkins identified conditions~\cite[Proposition 6.7]{devhop04} under which the $K(n)$-local $E$-Adams spectral sequence takes the form
\[
E_2^{\ast,\ast}=H_c^*(\G_n,E_*Z) \Rightarrow \pi_*L_{K(n)}Z. 
\]
It was remarked that this was probably not the most general result. In many cases the $E_2$-term of Adams-type spectral sequences can be calculated by $\Ext$ groups (for example~\cite[Chapter 2]{ravgreen}). Thus we ask the following two questions:
\begin{enumerate}[(a)]
	\item Can the $E_2$-term of the $K(n)$-local $E$-Adams spectral sequence be calculated by a suitable $\Ext$ group?
	\item In what generality can we identify the $E_2$-term with continuous group cohomology?
\end{enumerate}

In this document we give partial answers to both these questions. Some work on the second problem has been done previously, and we provide a comparison between some known results and our results. 

In the $K(n)$-local setting the natural functor to consider for a spectrum $X$ is not $E_*X$ but rather $\mE X \coloneqq \pi_*L_{K(n)}(E \wedge X)$. The use of this completed version of $E$-homology becomes very important in understanding the $E_2$-term of this spectral sequence.\footnote{This also gives one reason why the case of continuous cohomology with coefficients in $E^*Z$ is easier than in $\mE Z$ for any spectrum $Z$; since $F(Z,\Sigma^k E)$ is already $K(n)$-local for any $k \in \Z$ (since $E$ is), there is no need for a `completed' version of $E$-cohomology.} 
This is not just an $E_*$-module, but rather an $L$-complete $E_*$-module, and based on work of Baker~\cite{bakerlcomplete} we work in the category of $L$-complete $\mE E$-comodules. This category is not abelian, and so we use the methods of relative homological algebra to define a relative $\Ext$ functor for certain classes of objects in the category, which we denote by $\wExt_{\mE E}^{s,t}(-,-)$. The following is our answer for (a). 
\theoremstyle{plain}
\newtheorem*{thm:ass}{\Cref{thm:ass}}

\begin{thm:ass}
     Let $X$ and $Y$ be spectra and suppose that $\mE X$ is pro-free, and $\mE Y$ is either a finitely-generated $E_*$-module, pro-free, or has bounded $\frak{m}$-torsion (i.e., is annihilated by some power of $\frak{m}$). Then the $E_2$-term of the $K(n)$-local $E$-Adams spectral sequence with abutment $\pi_{t-s}F(X,L_{K(n)}Y)$ is
    \[
E_2^{s,t} =\wExt^{s,t}_{\mE E}(\mE X,\mE Y). 
    \]

 \end{thm:ass}

Our answer to question (b) is the following.
\theoremstyle{plain}
\newtheorem*{thm:COR}{\Cref{thm:COR}}
\begin{thm:COR}
Suppose that $X$ is a spectrum such that $\mE X$ is either a finitely-generated $E_*$-module, pro-free, or has bounded $\frak{m}$-torsion. 
Then, for the $K(n)$-local $E$-Adams spectral sequence with abutment $\pi_*L_{K(n)}X$, there is an isomorphism
    \[
E_2^{\ast,\ast} = \wExt^{\ast,\ast}_{\mE E}(E_*,\mE X) \simeq H_c^\ast(\G_n,\mE X).
    \]
\end{thm:COR}
We then compare this to some of the known results in the literature.

We have two applications of these results. Firstly, we can almost immediately extend a result of Goerss--Henn--Mahowald--Rezk, used in their construction of a resolution of the $K(2)$-local sphere at the prime 3~\cite{GHMR}, from finite subgroups of $\G_n$ to arbitrary closed subgroups. 

The second application appears to work at height $n=1$ only. Here we construct a spectral sequence with $E_2$-term $L_i\Ext^{s,t}_{E_*E}(E_*X,E_*Y)$, where $E_*X$ is a projective $E_*$-module, $E_*Y$ a flat $E_*$-module, and $L_i$ refers to the derived functor of completion on the category of $\Zp$-modules.  We show that the abutment of this spectral sequence is $\wExt_{\mE E}^{s-i,t}(\mE X,\mE Y)$ and calculate this when $X = Y = S^0$ at the prime 2. 
 \section*{Acknowledgements}
This work is an extension of part of the second author's PhD thesis at Melbourne University, supervised by Craig Westerland. We thank Craig for his tireless help and encouragement. We also thank Daniel Davis for helpful comments on an early draft of this paper, as well as Mark Hovey and Gabriel Valenzuela for helpful discussions. We thank the referee for numerous helpful comments and for catching a number of careless errors in an earlier version of this paper. The second author was supported by a David Hay postgraduate writing-up award provided by Melbourne University. The authors thank the Max Planck Institute for Mathematics in Bonn for its support and hospitality. 
\section{$L$-completion and $L$-complete comodules}
\subsection{$L$-completion}\label{sec:lcompletion}
It is now well understood (see, for example~\cite{hs99}) that in the $K(n)$-local setting the functor $\mE(-) = \pi_*L_{K(n)}(E \wedge -)$, from spectra to $E_*$-modules, mentioned in the introduction is a more natural covariant analogue of $E^*(-)$ than ordinary $E_*$-homology, despite the fact that it is not a homology theory. It is equally well understood that this functor is naturally thought of as landing in the category $\widehat{\text{Mod}}_{E_*}$ of $L$-complete $E_*$-modules, rather than the category of $E_*$-modules. We review the basics of this category now; for more details see~\cite{hs99,barthel2013completed,hovey2004ss,rezk2013}.

\begin{rem}\label{rem:bousfieldlocalisation}
  Since we always work with $E$-modules there is some ambiguity to the type of Bousfield localisation we are using. Recall that $L_{K(n)}$ denotes Bousfield localisation with respect to Morava $K$-theory $K(n)$ on the category of spectra. Let $L_{K(n)}^E$ denote Bousfield localisation on the category of $E$-modules. Suppose now that $M$ is an $E$-module. Then by~\cite[Lemma 4.3]{bakerjean} there is an equivalence $L_{K(n)}M \simeq L^{E}_{K(n) \wedge E}M$. But by~\cite[Proposition 2.2]{hovey08fil} the latter is just $L^E_{K(n)}M$ and so it does not matter if we use $L_{K(n)}$ or $L_{K(n)}^E$. 

\end{rem}

To keep the theory general, suppose that $R$ is a complete local Noetherian graded ring with a unique maximal homogeneous ideal $\mathfrak{m}$, generated by a regular sequence of $n$ homogeneous elements. Our assumptions imply that the (Krull) dimension of $R$ is $n$. Let $\operatorname{Mod}_R$ denote the category of graded $R$-modules, where the morphisms are the morphisms of $R$-modules that preserve the grading.

Recall that given an $R$-module $M$, the completion of $M$ (at $\mathfrak{m}$) is
\[
M^\wedge_\frak{m} = \varprojlim_k M/\mathfrak{m}^kM. 
\]
Here we must take the limit in the graded sense. This is functorial, but completion is not right (nor in fact left) exact; the idea is to then replace completion with its zeroth derived functor.
\begin{definition}
For $s \ge 0$ let $L_s(-):\operatorname{Mod}_R \to \operatorname{Mod}_R$ be the $s$-th left derived functor of the completion functor $(-)^\wedge_\mathfrak{m}$. 
\end{definition}
Since completion is not right exact it is \emph{not} true that $L_0M \simeq M^\wedge_\mathfrak{m}$. In fact the natural map $M \to M^\wedge_\mathfrak{m}$ factors as the composite $M \xr{\eta_M} L_0M  \xr{\epsilon_M} M^\wedge_\mathfrak{m}$.

\begin{definition}
We say that $M$ is $L$-complete if $\eta_M$ is an isomorphism of $R$-modules. 
\end{definition}
The map $\epsilon_M$ is surjective with kernel
\begin{equation}\label{eq:kernel}
\varprojlim{}^1_k \Tor_1^{R}(R/\mathfrak{m}^k,M);
\end{equation}
in general these derived functors fit into an exact sequence~\cite[Theorem A.2]{hs99} 
\[
0 \to \varprojlim{}^1_k \Tor^R_{s+1}(R/\mathfrak{m}^k,M) \to L_sM \to \varprojlim_k \Tor^R_s(R/\mathfrak{m}^k,M) \to 0,
\]
and vanish if $s<0$ or $s>n$.

Let $\Mod_R$ denote the subcategory of $\operatorname{Mod}_R$ consisting of those graded $R$-modules $M$ for which $\eta_M$ is an isomorphism. This category is a bicomplete full abelian subcategory of the category of graded $R$-modules, and is closed under extensions and inverse limits formed in $\operatorname{Mod}_R$. One salient feature of this category is that $\Ext^s_{\Mod_R}(M,N) \simeq \Ext^s_{R}(M,N)$ for all $s \ge 0$ whenever $M$ and $N$ are $L$-complete $R$-modules~\cite[Theorem 1.11]{hovey2004ss}. The tensor product of $L$-complete modules need not be $L$-complete; we write $M \btimes_R N \coloneqq L_0(M \otimes_R N)$. By~\cite[Proposition A.6]{hs99} this gives $\Mod_R$ the structure of a symmetric monoidal category. 

\begin{rem}\label{rem:tamemodules}

We will use the following properties of $L$-completion repeatedly:
\begin{enumerate}[(i)]
   \item If $M$ is a flat $R$-module, then $L_0M = M^\wedge_\mathfrak{m}$ is flat as an $R$-module and thus $L_sM = 0$ for $s>0$ (see~\cite[Corollary 1.3]{hovey2004ss} or~\cite[Proposition A.15]{barthel2013completed});
   \item If $M$ is a finitely-generated $R$-module, then $L_0M = M$ and $L_sM = 0$ for $s>0$~\cite[Proposition A.4, Theorem A.6]{hs99}; and,
   \item If $M$ is a bounded $\frak{m}$-torsion module, then $L_0M = M$ and $L_sM = 0$ for $s>0$. 
 \end{enumerate}  
 The last item follows from~\cite[Theorem A.6]{hs99} and the observation that for large enough $k$ there are equivalences (by~\cite[Proposition A.4]{hs99})
    \begin{equation*}
        \begin{split}
            L_0M & \simeq L_0(M \otimes_{R} R/\frak{m}^k) \\
            & \simeq L_0M \otimes_R R/\frak{m}^k \\
            & \simeq M \otimes_R R/\frak{m}^k \\
            & \simeq M,
        \end{split}
    \end{equation*}
    so that $M$ is $L$-complete. Modules $M$ that have $L_sM = 0$ for $s>0$ are known as \emph{tame}. For example, $L$-complete modules are always tame.  
    \end{rem}
\begin{example}\label{examp:zp}
    Let $R = \Zp$ and $\frak{m} = (p)$. Since $\Zp$ has Krull dimension 1 the only potential non-zero derived functors are $L_0$ and $L_1$. By~\cite[Proposition 5.2]{barthel2013completed}, $L$-completion with respect to $\Zp$ naturally lands in the category of $\Z_p$-modules.

It is immediate from the remark above that $L_0\Zp = \Z_p$ and $L_i\Zp =0$ for $i>0$.  By~\cite[Theorem A.2]{hs99} for any $\Zp$-module $M$ we have
\[
 L_0M = \Ext^1_{\Zp}(\Z/p^\infty,M)\quad \text{ and } \quad L_1M = \Hom_{\Zp}(\Z/p^\infty,M) \simeq \varprojlim_r \Hom_{\Zp}(\Z/p^r,M).
\]
If $M$ is any injective $\Zp$-module $M$, for example if $M =\Q/\Zp$, then it follows from this description that $L_0M = 0$.  
On the other hand the inverse system defined above gives $L_1(\Q/\Zp) = \Z_p$.

The discussion above also shows that if $M$ is any bounded $p$-torsion $\Zp$-module then it is $L$-complete and hence tame. 
\end{example}
Suppose now that $M$ is a flat $R$-module so that, by Lazard's theorem, we can write it canonically as a filtered colimit over finite free modules, $M = \varinjlim_j F_j$. Since $\Hom_R(F_j,L_0N)$ is $L$-complete for any $j \in J$, the same is true for $\varprojlim_j \Hom_R(F_j,L_0N) = \Hom_R(M,L_0N)$, and hence we get a natural factorization
\[\xymatrix{& L_0\Hom_R(M,N) \ar@{-->}[d] \\
\Hom_R(M,N) \ar[ru] \ar[r] & \Hom_R(M,L_0N)}\]
for arbitrary $N$. 
\begin{prop}
If $M$ is projective and $N$ is flat, then the natural map
\[\xymatrix{L_0\Hom_R(M,N) \ar[r] & \Hom_R(M,L_0N)}\]
is an isomorphism of $L$-complete $R$-modules.
\end{prop}

\begin{proof}
It is enough to show the claim for $M = \bigoplus_I R$ free. Since $R$ is Noetherian, products of flat modules are flat, so we get
\[L_0\Hom_R(M,N) = L_0\prod_I N = \varprojlim_k ((\prod_I N) \otimes R/\fm^k)\]
and similarly
\[\Hom_R(M,L_0N) = \prod_I \varprojlim_k (N \otimes R/\fm^k) = \varprojlim_k \prod_I (N \otimes R/\fm^k).\]
Therefore, it suffices to show that the natural map
\[ \epsilon: (\prod_I N) \otimes R/\fm^k \to \prod_I (N \otimes R/\fm^k)\]
is an isomorphism for all $k$. Since $R/\fm^k$ is finitely-presented, this is true by~\cite[Proposition 4.44]{MR1653294}, and the proposition follows. 
\end{proof}
\begin{cor}\label{cor:homlcompletion}
For $M$ projective and $N$ flat, there are isomorphisms
\[L_s\Hom_R(M,N) = 
\begin{cases}
\Hom_{\operatorname{\widehat{Mod}}_R}(L_0M,L_0N) & \text{if } s=0 \\
0 & \text{otherwise.}
\end{cases}\]
\end{cor}
\begin{proof}
The first statement is a direct consequence of the previous proposition. For the case of $s>0$ note that $\Hom_R(M,N)$ is flat, hence tame. 
\end{proof}
\begin{rem}
	Using work of Valenzuela~\cite{valenzuela} it is possible to construct a spectral sequence 
	\[
E_2^{s,t} = L_p \Ext^q_R(M,N) \Rightarrow \Ext_R^{p+q}(M,L_{R/\fm}N),
	\]
where $M$ and $N$ are arbitrary $R$-modules and $L_{R/\fm}$ is the total left derived functor of $L_0$. Specialising to $M$ projective and $N$ tame gives the above corollary.
\end{rem}
\subsection{Completed $E$-homology}
We now specialise to the case where $R = E_*$. By~\cite[Proposition 8.4]{hs99} the functor $\mE(-)$ always takes values in $\Mod_{E_*}$. This is in fact a special case of the following theorem.
\begin{prop}\cite[Corollary 3.14]{barthel2013completed}\label{prop:lcompletion}
	An $E$-module $M$ is $K(n)$-local if and only if $\pi_*M$ is an $L$-complete $E_*$-module.
\end{prop}
\begin{rem}
The case where $M = E \wedge X$, for $X$ an arbitrary spectrum, proved in~\cite{hs99}, uses a different method. In particular there is a tower of generalised Moore spectra $M_I$ such that $L_{K(n)}X \simeq \holim_I L_nX \wedge M_I$~\cite[Proposition 7.10]{hs99}. This gives rise to a Milnor sequence
\begin{equation}\label{eq:milnor}
0 \to \varprojlim_I{}^1 E_{\ast +1}(X \wedge M_I) \to \mE X \to \varprojlim_I E_\ast (X \wedge M_I) \to 0,
\end{equation}
which by~\cite[Theorem A.6]{hs99} implies $\mE X$ is $L$-complete.
\end{rem}

The projective objects in $\Mod_{E_*}$ will be important for us. These are characterised in~\cite[Theorem A.9]{hs99} and~\cite[Proposition A.15]{barthel2013completed}.  
\begin{definition}
An $L$-complete $E_*$-module is pro-free if it is isomorphic to the completion (or, equivalently, $L$-completion) of a free $E_*$-module. Equivalently, these are the projective objects in $\Mod_{E_*}$.
\end{definition}
 
\begin{prop}\label{prop:completemodules}
If $\mE X$ is either finitely-generated as an $E_*$-module, pro-free, or has bounded $\frak{m}$-torsion, then $\mE X$ is complete in the $\frak{m}$-adic topology. 
\end{prop}
\begin{proof}
   The case where $\mE X$ is finitely-generated follows from the fact that $E_*$ is complete and Noetherian. Since $\mE X$ is always $L$-complete and $L_0$-completion is idempotent, when $\mE X$ is pro-free (and hence flat) $L_0(\mE X) \simeq \mE X = (\mE X)^{\wedge}_\frak{m}$, so that $\mE X$ is complete. The case where $\mE X$ has bounded $\frak{m}$-torsion is clear. 
\end{proof}

\begin{rem}\label{rem:ktheory}
    The condition that $\mE X$ is pro-free is not overly restrictive. Let $K$ denote the 2-periodic version of Morava $K$-theory with coefficient ring $K_* = E_*/\frak{m} = \F_{p^n}[u^{\pm 1} ]$.  If $K_*X$ is concentrated in even degrees, then $\mE X$ is pro-free~\cite[Proposition 8.4]{hs99}. For example, this implies that $\mE E_n^{hF}$ is pro-free for any closed subgroup $F \subset \G_n$. By~\cite[Theorem 8.6]{hs99} $\mE X$ is finitely generated if and only if $X$ is $K(n)$-locally dualisable. 
\end{rem}

We will need the following version of the universal coefficient theorem (for $Y=S$ this is~\cite[Corollary 4.2]{hovey2004ss}). 
\begin{prop}\label{lem:UCT}
    Let $X$ and $Y$ be spectra. If $\mE X$ is pro-free, then
    \[
    \Hom_{E_*}(E^\vee_{\ast} X,E^\vee_{\ast} Y) \simeq \pi_*F(X,L_{K(n)}(E \wedge Y)).
    \]
\end{prop}

\begin{proof}
    Let $M,N$ be $K(n)$-local $E$-module spectra. Note that $\pi_*M$ and $\pi_*N$ are always $L$-complete by~\Cref{prop:lcompletion}. Under such conditions Hovey~\cite[Theorem 4.1]{hovey2004ss}
     has constructed a natural, strongly and conditionally convergent, spectral sequence of $E_*$-modules\footnote{Note that we have regraded the spectral sequence in~\cite{hovey2004ss} to reflect the fact we use homology rather than cohomology.}
    \[
    E_2^{s,t} = \Ext^{s,t}_{\Mod_{E_*}}(\pi_{*}M,\pi_{*}N) \simeq \Ext_{E_*}^{s,t}(\pi_{*}M,\pi_{*}N) \Rightarrow \pi_{t-s}F_E(M,N).
    \]
    Set $M = L_{K(n)}(E \wedge X)$ and $N = L_{K(n)}(E \wedge Y)$. Note then that
    \[
    F_E(L_{K(n)}(E \wedge X),L_{K(n)}(E \wedge Y)) \simeq F_E(E \wedge X,L_{K(n)}(E \wedge Y)) \simeq F(X,L_{K(n)}(E \wedge Y)),
    \]
where the second isomorphism is~\cite[Corollary III.6.7]{EKMM},
    giving a spectral sequence
    \[
    E_2^{s,t} = \Ext^{s,t}_{\Mod_{E_*}}(E^\vee_{\ast} X,E^\vee_{\ast} Y) \simeq \Ext_{E_*}^{s,t}(E^\vee_{\ast} X,E^\vee_{\ast} Y) \Rightarrow \pi_{t-s}F(X,L_{K(n)}(E \wedge Y)).
    \]
    Since $\mE X$ is pro-free it is projective in $\Mod_{E_*}$ and so the spectral sequence collapses, giving the desired isomorphism. 
\end{proof}
\begin{rem}
    The map above can be described in the following way: given 
    \[
f:X \to L_{K(n)}(E \wedge Y)
    \]
    then the homomorphism takes
    \[
g:S \to L_{K(n)}(E \wedge X)
    \]
    to the element
    \[
S \xr{g} L_{K(n)}(E \wedge X) \xr{1 \wedge f} L_{K(n)}(E \wedge E \wedge Y) \xr{\mu \wedge 1} L_{K(n)} (E\wedge Y).
    \]
\end{rem}

\subsection{$L$-complete Hopf algebroids}
Since $\mE X$ always lands in the category of $L$-complete $E_*$-modules, one is led to wonder if $\mE X$ is a comodule over a suitable $L$-complete Hopf algebroid. The category of $L$-complete Hopf algebroids has previously been studied by Baker~\cite{bakerlcomplete}, and we now briefly review this work.

Suppose that $R$ is as in~\Cref{sec:lcompletion} and, additionally, $R$ is an algebra over some local subring $(k_0,\frak{m}_0)$ of $(R,\frak{m})$, such that $\frak{m}_0 = k_0 \cap \frak{m}$. 

We say $A \in \Mod_{k_0}$ is a ring object if it has an associative product $\phi:A \otimes_{k_0} A \to A$. An $R$-unit for $\phi$ is a $k_0$-algebra homomorphism $\eta:R \to A$. A ring object $A$ is $R$-biunital if it has two units $\eta_L,\eta_R:R \to A$ which extend to give a morphism $\eta_L \otimes \eta_R: R \otimes_{k_0} R \to A$. Such an object is called $L$-complete if it is $L$-complete as both a left and right $R$-module.

\begin{definition}\cite[Definition 2.3]{bakerlcomplete}
Suppose that $\Gamma$ is an $L$-complete commutative $R$-biunital ring object with left and right units $\eta_L,\eta_R:R \to \Gamma$, along with the following maps:
\begin{align*}
\Delta&: \Gamma \to \Gamma \btimes_R \Gamma \text{ (composition) } \\
\epsilon&: \Gamma \to  R \text{ (identity) } \\
c&: \Gamma \to \Gamma \text{ (inverse) }
\end{align*}
satisfying the usual identities (as in~\cite[Appendix A]{ravgreen}) for a Hopf algebroid. Then the pair $(R,\Gamma)$ is an \emph{$L$-complete Hopf algebroid }if $\Gamma$ is pro-free as a left $R$-module, and the ideal $\frak{m}$ is invariant, i.e., $\frak{m}\Gamma = \Gamma \frak{m}$. 
\end{definition} 
\begin{lemma}~\cite[Proposition 5.3]{bakerlcomplete} 
The pair $(E_*,\mE E)$ is an $L$-complete Hopf algebroid. 
\end{lemma}

\begin{definition}~\cite[Definition 2.4]{bakerlcomplete}
Let $(R,\Gamma)$ be an $L$-complete Hopf algebroid. A left $(R,\Gamma)$-comodule $M$ is an $L$-complete $R$-module $M$ together with a left $R$-linear map $\psi:M \to \Gamma \btimes_R M$ which is counitary and coassociative.     
\end{definition}
We will usually refer to a left $(R,\Gamma)$-comodule as an $L$-complete $\Gamma$-comodule and we write $\wComod_\Gamma$ for the category of such comodules.

\begin{rem}
    In all cases we will consider, $\mE X$ will be a complete $E_*$-module, and so we could work in the category of complete $\mE E$-comodules, as studied previously by Devinatz~\cite{mordevinatz}. However, whilst the category of $L$-complete $E_*$-modules is abelian, the same is \emph{not} true for the category of complete $E_*$-modules, so we prefer to work with $L$-complete $\mE E$-comodules.
\end{rem}

Given an $L$-complete $R$-module $N$, let $\Gamma \btimes_R N$ be the comodule with structure map $\psi = \Gamma \btimes_R \Delta$. This is called an \emph{extended $L$-complete $\Gamma$-comodule}. The following is the standard adjunction between extended comodules and ordinary modules. 
\begin{lemma}\label{lem:adjoint}
    Let $N$ be an $L$-complete $R$-module and let $M$ be an $L$-complete $\Gamma$-comodule. Then there is an isomorphism
    \[
\Hom_{\operatorname{\widehat{Mod}}_R}(M,N) = \Hom_{\wComod_\Gamma}(M,\Gamma \btimes_R N). 
    \]
\end{lemma}

Suppose that $F$ is a ring spectrum (in the stable homotopy category) such that $F_*F$ is a flat $F_*$-module. In this case the pair $(F_*,F_*F)$ is an (ordinary) Hopf algebroid. To show that $F_*(X)$ is an $F_*F$-comodule for any spectrum $X$ requires knowing that $F_*(F \wedge X) \simeq F_*F \otimes_{F_*} F_*X$. The same is true here; to show that $\mE X$ is an $L$-complete $\mE E$-comodule we need to show that $\mE (E \wedge X) \simeq \mE E \btimes_{E_*} \mE X$. We do not know if it is true in general; our next goal will be to give the examples of $L$-complete $\mE E$-comodules that we need. We first start with a preliminary lemma.
\begin{lemma}\label{lem:tameness}
Let $M$ and $N$ be $E_*$-modules such that $M$ is flat and $N$ is either a finitely-generated $E_*$-module, pro-free, or has bounded $\frak{m}$-torsion. Then $M \otimes_{E_*} N$ is tame. 
\end{lemma}
\begin{proof}
  First assume $N$ is finitely-generated. Since $E_*$ is Noetherian there is a short exact sequence
  \[
0 \to K \to F \to N \to 0
  \]
  where $F = \oplus_I E_*$ is free and $K$ and $F$ are finitely-generated. Tensoring with the flat module $M$ gives another short exact sequence, and by~\cite[Theorem A.2]{hs99} there is a long exact sequence
  \begin{equation}\label{eq:les}
\cdots \to L_{k+1}(M \otimes_{E_*} N) \to L_k(M \otimes_{E_*}K) \to L_k(M \otimes_{E_*} F) \to L_k(M \otimes_{E_*} N) \to \cdots .\end{equation}
The functors $L_k$ are additive for all $k \ge 0$, and since $M$ is flat we see that $L_0(M \otimes_{E_*} F) = \oplus_I M^\wedge_\frak{m}$ and $L_k(M \otimes_{E_*}F) = 0$ for $k>0$. It follows that $L_{k+1}(M \otimes_{E_*}N) \simeq L_k(M \otimes_{E_*} K)$ for $k \ge 1$. 

Since $K,F$ and $N$ are all finitely-generated $E_*$-modules we use~\cite[Theorem A.4]{hs99} to see that the end of the long exact sequence~\eqref{eq:les} takes the form
 \[
0 \to L_1(M \otimes_{E_*}N) \to L_0(M) \otimes_{E_*} K \to L_0(M) \otimes_{E_*} F \to L_0(M) \otimes_{E_*} N \to 0.  
 \]
Since $M$ is flat, $L_0(M)$ is pro-free, and hence flat~\cite[Proposition A.15]{barthel2013completed}, so $L_0(M) \otimes_{E_*} K \to L_0(M) \otimes_{E_*} F$ is injective, forcing $L_1(M \otimes_{E_*}N) = 0$. Since $N$ was an arbitrary finitely-generated $E_*$-module and $K$ is finitely generated, we see that $L_1(M \otimes_{E_*} K) = 0$, also. It follows that $L_2(M \otimes_{E_*} N) \simeq L_1(M \otimes_{E_*} K) = 0$, and arguing inductively we see that $L_k(M \otimes_{E_*} N) = 0$ for $k>0$, so that $M \otimes_{E_*} N$ is tame. 

Now assume that $N$ is pro-free, and hence flat. It follows that $M \otimes_{E_*} N$ is also flat, and hence tame. 

For the final case, where $N$ has bounded $\frak{m}$-torsion, note that $M \otimes_{E_*} N$ also has bounded $\frak{m}$-torsion, and so is tame (see~\Cref{rem:tamemodules}). \qedhere

\end{proof}
We now identify conditions on a spectrum $X$ so that $\mE X$ is an $L$-complete $\mE E$-comodule. 
\begin{prop}\label{lem:comodules}
Let $X$ be a spectrum. If $\mE E \otimes_{E_*} \mE X$ is tame, then 
\begin{equation}\label{eq:kunneth}
\mE(E \wedge X) \simeq \mE E \btimes_{E_*} \mE X
\end{equation}
and $\mE X$ is an $L$-complete $\mE E$-comodule. In particular this occurs when $\mE X$ is either a finitely-generated $E_*$-module, pro-free or has bounded $\frak{m}$-torsion.
\end{prop}
\begin{proof}
There is a spectral sequence~\cite[Theorem IV.4.1]{EKMM}
\begin{equation}\label{eq:ekmm}
E^2_{s,t} = \Tor^{E_*}_{s,t}(\mE E,\mE X) \Rightarrow \pi_{s+t}(L_{K(n)}(E \wedge E)\wedge_E L_{K(n)}(E \wedge X)).
\end{equation}
For any $E$-module $M$ we also have the spectral sequence of $E_*$-modules~\cite[Theorem 2.3]{hovey08fil}
\[
E^2_{s,t} = (L_s\pi_*M)_t \Rightarrow \pi_{s+t} L_{K(n)}M.
\]
In particular there is a spectral sequence starting from the abutment of~\ref{eq:ekmm} that has the form
\[
(L_i\pi_{\ast}(L_{K(n)}(E \wedge E)\wedge_E L_{K(n)}(E \wedge X)))_{s+t} \Rightarrow \pi_{i+s+t}L_{K(n)}(L_{K(n)}(E \wedge E)\wedge_E L_{K(n)}(E \wedge X)).
\]
By~\Cref{rem:bousfieldlocalisation} we deduce that there is an equivalence
\[
  L_{K(n)}(L_{K(n)}(E \wedge E)\wedge_E L_{K(n)}(E \wedge X)) \simeq L_{K(n)}(E \wedge E \wedge X),
\]
and so the latter spectral sequence abuts to $\mE (E \wedge X)$. 
\color{black}
Since $\mE E$ is a flat $E_*$-module the first spectral sequence always collapses, and the second spectral sequence becomes
\begin{equation}\label{eq:ss}
(L_i(\mE E \otimes_{E_*} \mE X))_{s+t} \Rightarrow E^\vee_{i+s+t} (E \wedge X).
\end{equation}
Thus, if $\mE E \otimes_{E_*} \mE X$ is tame, this gives an isomorphism
\[
\mE(E \wedge X) \simeq \mE E \btimes_{E_*} \mE X,
\]
and so $\mE X$ is an $L$-complete $\mE E$-comodule. Since $\mE E$ is pro-free it is flat, and~\Cref{lem:tameness} applies to show that $\mE E \otimes_{E_*} \mE X$ is tame in the given cases. 
\end{proof}
\begin{rem}
This raises the question: what is the most general class of $L$-complete comodules $M$ such that $\mE E \otimes_{E_*} M$ is tame? In light of Baker's example~\cite[Appendix B]{bakerlcomplete} of an $L$-complete - and hence tame - module $N$ such that $L_1(\bigoplus_{i=0}^{\infty}N) \ne 0$, this seems to be a subtle problem. In particular, we note that this example implies that the collection of tame modules itself need not satisfy the above condition. 
\end{rem}
  The following corollary shows that the equivalence of~\eqref{eq:kunneth} can be iterated. 
    
\begin{cor}\label{lem:iteratedkunneth}
  Let $Y$ be a spectrum such that $\mE Y$ is either a finitely-generated $E_*$-module, pro-free or has bounded $\frak{m}$-torsion.  Then for all $s \ge 0$ there is an isomorphism
  \[
\mE (E^{\wedge s} \wedge Y ) \simeq (\mE E)^{\btimes s} \btimes_{E_*} \mE Y.
  \]
\end{cor}
\begin{proof}
We will prove this by induction on $s$, the case $s=0$ being trivial. Assume now that $\mE (E^{\wedge (s-1)} \wedge Y) \simeq (\mE E)^{\btimes (s-1)} \btimes_{E_*} \mE Y $; we will show that $\mE E \otimes_{E_*} ((\mE E)^{\btimes (s-1)} \btimes_{E_*} \mE Y)$ is tame. We claim that this is true in the three cases we consider. 
\begin{enumerate}
  \item If $\mE Y$ is flat, then so is $(\mE E)^{\btimes(s-1)}\btimes_{E_*}\mE Y$, and we can apply~\Cref{lem:tameness} to see that $\mE E \otimes_{E_*} ((\mE E)^{\btimes (s-1)} \btimes_{E_*} \mE Y)$ is tame. 
  \item If $\mE Y$ is finitely-generated then $(\mE E)^{\btimes(s-1)}\btimes_{E_*}\mE Y \simeq (\mE E)^{\btimes(s-1)}\otimes_{E_*}\mE Y$~\cite[Theorem A.4]{hs99}. Since $\mE E \otimes_{E_*} (\mE E)^{\btimes(s-1)}$ is a flat $E_*$-module, once again we can apply~\Cref{lem:tameness} to see that $\mE E \otimes_{E_*} ((\mE E)^{\btimes (s-1)} \btimes_{E_*} \mE Y)$ is tame.
  \item If $\mE Y$ has bounded $\frak{m}$-torsion, then the same is true for $\mE E \otimes_{E_*} ((\mE E)^{\btimes (s-1)} \btimes_{E_*} \mE Y)$, and it follows that it is tame, as required.
\end{enumerate}
Therefore, \Cref{lem:comodules} applied to $X =E^{\wedge (s-1)} \wedge Y$  implies that
\[
\mE(E^{\wedge s} \wedge Y) \simeq \mE E \btimes_{E_*} \mE(E^{\wedge (s-1)} \wedge Y) \simeq (\mE E)^{\btimes s} \btimes_{E_*} \mE Y,
\]
where the last isomorphism uses the inductive hypothesis once more. 
\end{proof}

\section{Relative homological algebra}

\subsection{Motivation}
Recall~\cite[Appendix A]{ravgreen} that the category of comodules over a Hopf algebroid $(A,\Gamma)$ is abelian whenever $\Gamma$ is flat over $A$, and that if $I$ is an injective $A$-module then $\Gamma \otimes_A I$ is an injective $\Gamma$-comodule. This implies that the category of $\Gamma$-comodules has enough injectives. 

Given $\Gamma$-comodules $M$ and $N$ we can define $\Ext^i_\Gamma(M,N)$ in the usual way as the $i$-th derived functor of $\Hom_\Gamma(M,N)$, functorial in $N$. However, the category of $L$-complete $\Gamma$-comodules does not need to be abelian. In this case, in order to define $L$-complete $\Ext$-groups, we need to use relative homological algebra, for which the following is meant to provide some motivation.

The following two lemmas show that we can form a resolution by~\emph{relative injective} objects, instead of absolute injectives. 
\begin{lemma}\label{lem:relres}
	Let $(A,\Gamma)$ be a Hopf algebroid (over a commutative ring $K$) such that $\Gamma$ is a flat $A$-module, and let
\[
	0 \to N \to R^0 \to R^1 \to \cdots
	\]
	be a sequence of left $\Gamma$-comodules which is exact (over $K$) and such that for each $i$, $\Ext^n_\Gamma(M,R^i) = 0$ for all $n>0$. Then $\Ext_\Gamma(M,N)$ is the cohomology of the complex
	\[
\Ext_\Gamma^0(M,R^0) \to \Ext_\Gamma^0(M,R^1) \to \cdots.
	\]
\end{lemma}
\begin{proof}
	See~\cite[Lemma 1.1]{mrw77} or~\cite[Lemma A1.2.4]{ravgreen}.
\end{proof}

\begin{definition}
	A $\Gamma$-comodule $S$ is a relative injective  $\Gamma$-comodule if it is a direct summand of an extended comodule, i.e., one of the form $\Gamma \otimes_A N$.
\end{definition}

\begin{lemma}\label{lem:proj}
	Let $S$ be a relatively injective comodule. If $M$ is a projective $A$-module, then $\Ext_\Gamma^i(M,S) = 0$ for $i>0$. Hence if $I^*$ is a resolution of $N$ by relatively injective comodules then 
    \begin{equation}\label{eq:extequiv}
\Ext^n_\Gamma(M,N) = H^n(\Hom_\Gamma(M,I^*))
    \end{equation}
    for all $n \ge 0$. 
\end{lemma}
\begin{proof}
The second statement follows from the first and~\Cref{lem:relres}. For the first statement proceed as in~\cite[A1.2.8(b)]{ravgreen}.
\end{proof}

In the case of $L$-complete $\Gamma$-comodules, we will take the analogue of~\Cref{eq:extequiv} as a definition of $\wExt_{\Gamma}(-,-)$ (see~\Cref{def:lcompleteext}).

\begin{rem}
    The reader may wonder about projective objects. In general, comodules over a Hopf algebra do not have enough projectives. For example, when $(A,\Gamma) = (\F_p,\mathcal{A})$, where $\mathcal{A}$ is the dual of the Steenrod algebra, it is believed that there are no nonzero projective objects~\cite{stablepalmeri}. 
\end{rem}
\subsection{Homological algebra for $L$-complete comodules}
The category $\wComod_\Gamma$ of $L$-complete $\Gamma$-comodules is not abelian; it is an additive category with cokernels. The absence of kernels is due to the failure of tensoring with $\Gamma$ to be flat. If $\theta: M \to N$ is a morphism of $L$-complete comodules, then there is a commutative diagram~\cite{bakerlcomplete}
\[
\begin{tikzcd}
     0 \arrow{r} & \ker \theta \arrow[dashed]{d} \arrow{r} & M \arrow{r}{\theta} \arrow{d}{\psi_M} & N \arrow{d}{\psi_N} \\
     & \Gamma \btimes_R \ker \theta \arrow{r} & \Gamma \btimes_R M \arrow{r}{\text{id} \btimes_R \theta} & \Gamma \btimes_R N,
 \end{tikzcd} 
\]
but the dashed arrow need not exist or be unique. 

Since $\wComod_\Gamma$ is not abelian we need to use the methods of relative homological algebra to define a suitable Ext functor, which we briefly review now. For a more thorough exposition see~\cite{eilenbergmoore} (although in general one needs to dualise what they say, since they mainly work with relative projective objects). Our work is in fact similar to that of Miller and Ravenel~\cite{mrw77}.

\begin{definition}\label{def:injclass}
An injective class $\mathcal{I}$ in a category $\mathcal{C}$ is a pair $(\mathcal{D},\mathcal{S})$ where $\mathcal{D}$ is a class of objects and $\mathcal{S}$ is a class of morphisms such that:
\begin{enumerate}
\item $I$ is in $\mathcal{D}$ if and only if for each $f:A \to B$ in $\mathcal{S}$
\[
f^\ast:\Hom_{\mathcal{C}}(B,I) \to \Hom_{\mathcal{C}}(A,I) 
\]
is an epimorphism. We call such objects relative injectives. 
\item A morphism $f:A \to B$ is in $\mathcal{S}$ if and only if for each $I \in \mathcal{D}$  
\[
f^\ast:\Hom_{\mathcal{C}}(B,I) \to \Hom_{\mathcal{C}}(A,I) 
\]
is an epimorphism. These are called the relative monomorphisms. 
\item\label{item:relinj} For any object $A \in \mathcal{C}$ there exists an object $Q \in \mathcal{D}$ and a morphism $f:A \to Q$ in $\mathcal{S}$. 
\end{enumerate}
\end{definition}
\begin{rem}
	Note that given either $\mathcal{S}$ or $\mathcal{D}$, the other class is determined by the requirements above, and that the third condition ensures the existence of enough relative injectives. 
\end{rem}
It is not hard to check that $\mathcal{D}$ is closed under retracts and that if the composite morphism $A \xr{f}B \to C$ is in $\mathcal{S}$ then so is $f:A \to B$.

\begin{example}[The split injective class]
 The \emph{split injective class} $\mathcal{I}_s = (\mathcal{D}_s,\mathcal{S}_s)$ has $\mathcal{D}_s$ equal to all objects of $\mathcal{C}$ and $\mathcal{S}_s$ all morphisms that satisfy~\Cref{def:injclass}, i.e., $\Hom_\mathcal{C}(f,-)$ is surjective for all objects. One can easily check that this is equivalent to the requirement that $f:A \to B$ is a split monomorphism. 

\end{example}

\begin{example}[The absolute injective class]
Let $\mathcal{S}$ be the class of all monomorphisms and then let $\mathcal{D}$ be the objects as needed. This satisfies (3) if there are enough categorical injectives. 
\end{example}

One way to construct an injective class is via a method known as reflection of adjoint functors. 
\begin{prop}
     Suppose that $\mathcal{C}$ and $\mathcal{F}$ are additive categories with cokernels, and there is a pair of adjoint functors
\[
T:\mathcal{C} \rightleftarrows \mathcal{F}:U.
\]
Then, if $(\mathcal{D},\mathcal{S})$ is an injective class in $\mathcal{C}$, we define an injective class $(\mathcal{D}',\mathcal{S}')$ in $\mathcal{F}$, where the class of objects is given by the set of all retracts of $T(\mathcal{D})$ and the class of morphisms is given by all morphisms whose image (under $U$) is in $\mathcal{S}$.
\end{prop}
\begin{proof}[Sketch of proof.\footnotemark] \footnotetext{For full details see~\cite[p.~15]{eilenbergmoore} - here it is proved for relative projectives, but it is essentially formal to dualise the given argument.}
 First note that, since relative injectives are closed under retracts, to show that $\mathcal{D'}$ is as claimed, it suffices to show that $T(I)$ is relative injective, whenever $I \in \mathcal{D}$. Let $A \to B$ be in $\mathcal{S}'$ and $I \in \mathcal{D}$; then the map
\[
\Hom_\mathcal{F}(B,T(I)) \to \Hom_\mathcal{F}(A,T(I))
\]
is equivalent to the epimorphism
\[
\Hom_{\mathcal{C}}(U(B),I) \to \Hom_{\mathcal{C}}(U(A),I).
\]
A similar method shows that the relative monomorphisms are as claimed. Finally we observe that for all $A \in \mathcal{F}$ there exists a $Q \in \mathcal{D}$ such that $U(A) \to Q \in \mathcal{S}$. Then the adjoint $A \to T(Q)$ satisfies Condition 3. To see this note that $U(A) \to Q$ factors as $U(A) \to U(T(Q)) \to Q$; since relative monomorphisms are closed under left factorisation (see above) $U(A) \to U(T(Q)) \in \mathcal{S}$. Then $A \to T(Q) \in \mathcal{S}'$ as required.
 \end{proof}
 We recall the following definition.
\begin{definition}
An extended $L$-complete $\mE E$-comodule is a comodule isomorphic to one of the form $\mE E \btimes_{E_*} M$, where $M$ is an $L$-complete $E_*$-module. Here the comultiplication is given by the map
\[
\mE E \btimes_{E_*} M \xr{\Delta \btimes \text{id}} \mE E\btimes_{E_*} \mE E \btimes_{E_*} M.
\]
\end{definition}
\begin{example}\label{examp:Morava}
Give $\Mod_{E_*}$ the split injective class. Then the adjunction
\[
\Hom_{\operatorname{\widehat{Mod}_{E_*}}}(A,B) = \Hom_{\wComod_{\mE E}}(A,\mE E \btimes_{E_*} B)
\]
produces an injective class in $\wComod_{\mE E}$. In particular we have 
\begin{enumerate}
	\item $\mathcal{S}$ is the class of all comodule morphisms $f: A \to B$ whose underlying map of $L$-complete $E_*$-modules is a split monomorphism.
	\item  $\mathcal{D}$ is the class of $L$-complete $\mE E$-comodules which are retracts of extended complete $\mE E$-comodules. 
\end{enumerate}
 Note that for any complete $\mE E$-comodule $M$ the coaction map $M \xr{\psi} \mE E\htimes_{E_*} M$ is a relative monomorphism into a relative injective. 

\end{example}
We will say that a three term complex $M \xr{f} N \xr{g} P$ of comodules is relative short exact if $gf = 0$ and $f:M \to N$ is a relative monomorphism. A relative injective resolution of a comodule $M$ is a complex of the form
\[
0 \to M \to J^0 \to J^1 \to \cdots
\]
where each $J^i$ is relatively injective, and each three-term subsequence 
\[
J^{s-1} \to J^s \to J^{s+1},
\]
where $J^{-1} = M$ and $J^s = 0$ for $s< -1$, is relative short exact.  Note that, by definition, relative exact sequences are precisely those that give exact sequences of abelian groups after applying $\Hom_{\wComod_{\mE E}}(-,I)$, whenever $I$ is relative injective. 

We have the usual comparison theorem for relative injective resolutions. The proof is nearly identical to the standard inductive homological algebra proof - in this context see~\cite[Theorem 2.2]{husmoore}.
\begin{prop}\label{prop:comparasion}
	Let $M$ and $M'$ be objects in an additive category $\mathcal{C}$ with relative injective resolutions $P^\ast$ and $P'^\ast$, respectively. Suppose there is a map $f: M \to M'$. Then, there exists a chain map $f^{\ast}:P^\ast \to P'^\ast$ extending $f$ that is unique up to chain homotopy. 
\end{prop}

\begin{definition}(cf.~\cite[p.~7]{eilenbergmoore}.)\label{def:lcompleteext}
Let $M$ and $N$ be $L$-complete $\mE E$-comodules, and let $M$ be pro-free. Let $I^\ast$ be a relative injective resolution of $N$. Then, for all $s \ge 0$, we define
\[
\wExt_{\wComod_{\mE E}}^s(M,N) = H^s(\Hom_{\wComod_{\mE E}}(M,I^\ast)).
\]
\end{definition}
For brevity we will write $\wExt_{\mE E}^s(M,N)$ for this Ext functor. 

Note that~\Cref{prop:comparasion} implies that the derived functor is independent of the choice of relative injective resolution. 
\begin{rem}\begin{enumerate}
	\item The reader should compare this definition to~\Cref{lem:proj}.
	\item The category of $L$-complete $E_*$-modules has no non-zero injectives~\cite[p.~40]{barthel2013completed}; this suggests that the same is true of $L$-complete $\mE E$-comodules, which is yet another reason we need to use relative homological algebra. 
\end{enumerate}

\end{rem}
Let $M$ be an $L$-complete comodule. As in~\cite{mrw77} we have the standard, or cobar resolution of $M$, denoted $\Omega^*(\mE E,M)$, with
\[
\Omega^n(\mE E,M) = \underbrace{\mE E \btimes_{E_*} \cdots \btimes_{E_*} \mE E}_{n+1 \text{ times }} \btimes_{E_*}  M
\]
and differential
\begin{equation*}
\begin{split}
d(e_0 \btimes \cdots \btimes e_n \btimes m) &= \sum_{i=0}^n (-1)^i e_0 \btimes \cdots e_{i-1} \btimes \Delta(e_i) \btimes e_{i+1} \btimes \cdots \btimes m \\
& + (-1)^{n+1} e_0 \btimes \cdots e_n \btimes \psi(m).
\end{split}
\end{equation*}
The usual contracting homotopy of~\cite{mrw77} given by
\[
s(e_0 \btimes \cdots \btimes e_n \btimes m) = \epsilon(e_0) e_1 \btimes \cdots \btimes e_n \btimes m
\] shows that $(\Omega^*(\mE E,M),d)$ defines a relative injective resolution of $M$.

\begin{lemma}\label{lem:ext0}
Let $M$ and $N$ be $L$-complete $\mE E$-comodules. Then there is an isomorphism
\[
\wExt_{\mE E}^0(M,N) \simeq \Hom_{\wComod_{\mE E}}(M,N). 
\]
\end{lemma}
\begin{proof}
  Let $\psi_M:M \to \mE E \btimes_{E_*} M$ and $\psi_N:N \to \mE E \btimes_{E_*} N$ be the comodule structure maps. Define
  \[
\psi_M^\ast,\psi_N^\ast:\Hom_{\Mod_{E_*}}(M,N) \to \Hom_{\Mod_{E_*}}(M,\mE E \btimes_{E_*} N)
  \]
by
\[
\psi_M^\ast(f) = (1 \btimes f)\psi_M \quad \text{ and } \quad \psi_N^\ast(f) = \psi_Nf.
\]

    Note that (see~\cite[Proof of A1.1.6]{ravgreen} for the case of an ordinary Hopf algebroid) \[
    \Hom_{\wComod_{\mE E}}(M,N) = \ker(\psi_M^\ast - \psi_N^\ast).
    \]
The cobar complex begins
\[
\xymatrix{
  \mE E \ar[r] & \mE E \btimes_{E_*} N \ar[rr]^-{\Delta \btimes 1 - 1 \btimes \psi_N} && \mE E \btimes_{E_*} \mE E \btimes_{E_*} N. 
}
\]
Applying $\Hom_{\wComod_{\mE E}}(M,-)$ and using the adjunction of~\Cref{lem:adjoint} between extended $L$-complete $\mE E$-comodules and $L$-complete $E_*$-modules we see that
\[
\wExt_{\mE E}^0(M,N) = \ker\left(\Hom_{\Mod_{E_*}}(M,N) \xr{f} \Hom_{\Mod_{E_*}}(M,\mE E \btimes_{E_*} N)\right). 
\]
One can check that the map $f$ is precisely $\psi_M^\ast - \psi_N^\ast$, and the claim follows. 
\end{proof}
\section{The $K(n)$-local $E_n$-Adams spectral sequence}
\subsection{Adams spectral sequences}
    Here we present some standard material on Adams-type spectral sequences following~\cite{millerrelations,devhop04}. Throughout this section we always work in the homotopy category of spectra. 

    Let $R$ be a ring spectrum. We say that a spectrum $I$ is $R$-injective if the map $I \to R \wedge I$ induced by the unit is split. A sequence of spectra $X' \to X \to X''$ is called $R$-exact if the composition is trivial and
    \[
[X',I] \leftarrow [X,I] \leftarrow [X'',I]
    \]
    is exact as a sequence of abelian groups for each $R$-injective spectrum $I$. An $R$-resolution of a spectrum $X$ is then an $R$-exact sequence of spectra (i.e., each three term subsequence is $R$-exact) 
    \[
\ast \to X \to I^0 \to I^1  \to \cdots
    \]
    such that each $I^s$ is $R$-injective. Given an $R$-resolution of $X$ we can always form an Adams resolution of $X$; that is, a diagram
    \[
    \begin{tikzcd}
        X=X_0\arrow[swap]{dr}{j} && \arrow[swap]{ll}{i} X_1  \arrow[swap]{dr}{j} && \arrow[swap]{ll}{i}  X_2 \arrow[swap]{dr}{j} && \arrow[swap]{ll}{i}  X_3 \\
        & I^0 \arrow[dashed,swap]{ur}{k} && \Sigma^{-1}I^1 \arrow[dashed,swap]{ur}{k} && \Sigma^{-2}I^2 \arrow[dashed,swap]{ur}{k}
    \end{tikzcd}
    \cdots
    \]
    such that each $\Sigma^{-s} I^s$ is $R$-injective and each $X_{k+1} \to X_k \to \Sigma^{-k}I^k$ is a fiber sequence. Note that the composition $I^k \to \Sigma^{k+1} X_{k+1} \to I^{k+1}$ corresponds to the original morphism in the $R$-resolution of $X$.

Given such a diagram we can always form the following exact couple
\[
\begin{tikzcd}
D^{s+1,t+1} = \pi_{t-s}(X_{s+1}) \arrow{rr}{i} && \pi_{t-s}(X_s) = D^{s,t} \arrow{dl}{j} & \\
& E_1^{s,t} = \pi_{t-s}(\Sigma^{-s}I^s). \arrow[dashed]{ul}{k}
\end{tikzcd}
\]
If we form the standard resolution, where $I^k = R^{\wedge(k+1)} \wedge X$ for $k \ge 0$, and if $R_*R$ is a flat $R_*$-module, then it is not hard to see that on the $E_1$-page we get the following sequence
\[
0 \to R_*X \to R_*R \otimes_{R_*} R_*X \to R_*R^{\otimes 2} \otimes_{R_*} R_*X \to \cdots.
\]
By explicitly checking the maps one can see that this is the cobar complex for computing $\Ext$, and so we get the usual Adams spectral sequence
\[
\Ext^{\ast,\ast}_{R_*R}(R_*,R_*X) \Rightarrow \pi_* X^\wedge_R.
\]
  
Here $X^\wedge_R$ is the $R$-nilpotent completion of $X$~\cite{bousfield79}. This construction can be suitably modified to construct the $F$-local $R$-Adams spectral sequence (see~\cite[Appendix A]{devhop04}), where $F$ is any spectrum. Following Devinatz and Hopkins say an $F$-local spectrum $I$ is $R$-injective if the map $I \to L_F(R \wedge I)$ is split. The definition of $R$-exact sequence and $R$-exact resolution then follow in the same way as the unlocalised case. 

We specialise to the case where $F = K(n)$ and $R = E$ is Morava $E$-theory. Following~\cite[Remark A.9]{devhop04} we take $I^j = L_{K(n)}(E^{\wedge (j+1)} \wedge X)$.  The formulas of~\cite[Construction 4.11]{devhop04} actually show that the $I^j$ form an Adams resolution (in fact they can be assembled into a cosimplicial resolution). Here is our main result.
\begin{theorem}\label{thm:ass}
     Let $X$ and $Y$ be spectra and suppose that $\mE X$ is pro-free, and $\mE Y$ is either a finitely-generated $E_*$-module, pro-free, or has bounded $\frak{m}$-torsion (i.e., is annihilated by some power of $\frak{m}$). Then the $E_2$-term of the $K(n)$-local $E$-Adams spectral sequence with abutment $\pi_{t-s}F(X,L_{K(n)}Y)$ is
    \[
E_2^{s,t} =\wExt^{s,t}_{\mE E}(\mE X,\mE Y). 
    \] 
\end{theorem}
\begin{proof}
    By mapping $X$ into an Adams resolution of $L_{K(n)}Y$ we obtain an exact couple with $E_1^{s,t} = \pi_{t-s}F(X,\Sigma^{-s}I^s) \simeq \pi_tF(X,I^s)$. Unwinding the exact couple we see that the $E_2$-page is the cohomology of the complex
    \[
\pi_*F(X,I^0) \to \pi_*F(X,I^1) \to \pi_*F(X,I^2) \to \cdots.
    \]

As usual, the Adams spectral sequence is independent of the choice of resolution from the $E_2$-page onwards, so we use the standard resolution, i.e., we let  $I^s = L_{K(n)}(E^{\wedge (s+1)} \wedge Y)$.  Applying~\Cref{lem:UCT} (which we can do under the assumption that $\mE X$ is pro-free) we see that
    \[
    \begin{split}
    \pi_\ast F(X,I^s)&\simeq \Hom_{E_*}(\mE X,\mE (E^{\wedge s} \wedge Y))     \\
    &\simeq \Hom_{\Mod_{E_*}}(\mE X,\mE (E^{\wedge s} \wedge Y)),
    \end{split}
\]
where the latter follows from the fact that $\mE(-)$ is always $L$-complete. 

By~\Cref{lem:iteratedkunneth} we have $\mE(E^{ \wedge s} \wedge Y) \simeq (\mE E)^{\btimes s} \btimes_{E_*} \mE Y$. Using the adjunction between extended comodules and $L$-complete $E_*$-modules we get
\[
 \Hom_{\Mod_{E_*}}(\mE X,\mE (E^{\wedge s} \wedge Y)) \simeq \Hom_{\wComod_{\mE E}}(\mE X, (\mE E)^{\btimes (s+1)} \btimes_{E_*} \mE Y).
\]
This implies that the $E_2$-page is the cohomology of the complex
\[
\Hom_{\wComod_{\mE E}}(\mE X,\mE E \btimes_{E_*} \mE Y) \to \Hom_{\wComod_{\mE E}}(\mE X,(\mE E)^{\btimes 2} \btimes_{E_*} \mE Y) \to \cdots
\]
which is precisely $\wExt^{\ast,\ast}_{\mE E}(\mE X,\mE Y)$. 
\end{proof}
The following is now a consequence of~\cite[Theorem 2]{devhop04} and uniqueness of the $E_2$-term.
\begin{cor}\label{cor:e2htfp}
    Let $X$ be a spectrum such that $\mE X$ is pro-free. If $F$ is a closed subgroup of $\G_n$, then there is an isomorphism
    \[
\wExt^{s,t}_{\mE E}(\mE X,\mE E^{hF}) \simeq H_c^s(F,\pi_tF(X,E)) \simeq H_c^s(F,E^{-t}X).
    \]
\end{cor}
\section{Identification of the $E_2$-term with group cohomology}\label{sec:morCOR}
It has been known since the work of Morava~\cite{mor85}, that completed $\Ext$ groups (as considered in~\cite{mordevinatz}) can, under some circumstances, be identified with continuous group cohomology. The results of this section say that our $L$-complete $\Ext$ groups can also be identified with continuous group cohomology. In fact, in many cases complete and $L$-complete $\Ext$ groups coincide, although we do not make this statement precise. Before we can give our result identifying $L$-complete $\Ext$ groups and group cohomology, we need two preliminary lemmas. We write $M \htimes_{E_*} N$ for the $\frak{m}$-adic completion of the ordinary tensor product. 
\begin{lemma}\label{lem:completefg}
  Let $M$ be a pro-free $E_*$-module, and $N$ a finitely-generated $E_*$-module. Then
  \[
M \htimes_{E_*}N \simeq M \otimes_{E_*} N.
  \]
\end{lemma}
\begin{proof}
For a fixed finitely generated $E_*$-module $N$, the category of $E_*$-modules $M$ for which the conclusion of the lemma holds is closed under retracts.
  By~\cite[Proposition A.13]{hs99} $M = L_0(\oplus_I E_*)$ is a retract of $\prod_I E_*$, and so it suffices to prove the lemma for $M=\prod_I E_*$. Note that by~\cite[Proposition 4.44]{MR1653294} an $E_*$-module $N$ is finitely-presented (equivalently, finitely-generated, since $E_*$ is Noetherian) if and only if for any collection $\{ C_\alpha \}$ of $E_*$-modules, the natural map  $N \otimes_{E_*} \prod C_\alpha \to \prod(N \otimes_{E_*} C_\alpha)$ is an equivalence. 

  We then have a series of equivalences
  \[
\begin{split}
  (\prod_I E_*) \htimes_{E_*} N &= \varprojlim_k (E_*/\frak{m}^k \otimes_{E_*} (\prod_I E_*) \otimes_{E_*} N) \ \\
  & \simeq \varprojlim_k \prod_I (E_*/\frak{m}^k \otimes_{E_*} N)\\
  & \simeq \prod_I \varprojlim (E_*/\frak{m}^k \otimes_{E_*} N) \\
  & \simeq \prod_I N^\wedge_{\frak{m}} \\
  & \simeq \prod_I N \\
  & \simeq (\prod_I E_*) \otimes_{E_*} N\qedhere. 
\end{split}
  \]
\end{proof}
\begin{lemma}\label{lemma:lcompletetotensor}
Let $M$ and $N$ be $E_*$-modules. Suppose that $M$ is pro-free and $N$ is either pro-free, finitely-generated as an $E_*$-module, or has bounded $\frak{m}$-torsion. Then
\[
M \htimes_{E_*} N \simeq M \btimes_{E_*} N.
\]
\end{lemma}
\begin{proof}
Note that in each case~\Cref{prop:completemodules} implies that $M$ and $N$ are both complete in the $\frak{m}$-adic topology.    When $N$ is finitely generated there is an isomorphism~\cite[Proposition A.4]{hs99}
    \[
M \btimes_{E_*} N \simeq L_0(M) \otimes_{E_*} N \simeq M \otimes_{E_*} N,
    \]
    where the last isomorphism follows from the fact that $M$ is already $L$-complete.  Since $N$ is finitely-generated~\Cref{lem:completefg} implies that $M \otimes_{E_*} N \simeq M \htimes_{E_*} N$.

  Now suppose that $N$ has bounded $\frak{m}$-torsion. Note that $M \otimes_{E_*} N$ is still bounded $\frak{m}$-torsion, and so $M \btimes_{E_*} N \simeq M \otimes_{E_*} N$. Furthermore, there is an isomorphism $M \otimes_{E_*} N \simeq M \htimes_{E_*} N$, since the inverse system defining the completed tensor product is eventually constant. 

    For the final case, assume that $N$ is pro-free. Since both $M$ and $N$ are flat $E_*$-modules the same is true for $M \otimes_{E_*} N$. This implies (see~\Cref{rem:tamemodules}) that $M \htimes_{E_*} N \simeq M \btimes_{E_*} N$. 
\end{proof}
We can now identify when the $E_2$-term of the $K(n)$-local $E$-Adams spectral sequence is given by continuous group cohomology. 
\begin{theorem}\label{thm:COR}
Suppose that $X$ is a spectrum such that $\mE X$ is either a finitely-generated $E_*$-module, pro-free, or has bounded $\frak{m}$-torsion. 
Then, for the $K(n)$-local $E$-Adams spectral sequence with abutment $\pi_*L_{K(n)}X$, there is an isomorphism
    \[
E_2^{s,t} = \wExt^{s,t}_{\mE E}(E_*,\mE X) \simeq H_c^s(\G_n,E^\vee_t X).
    \]
\end{theorem}
To prove this we will need to be explicit about the definition of continuous group cohomology we use, following~\cite[Section 2]{tate_k2}. Let $N$ be a topological $\G_n$-module and define
\[
C^k(\G_n,N) = \Hom^c(\G_n^k,N),
\]
the group of continuous functions from $\G_n^k$ to $N$, where $\G_n^k = \underbrace{\G_n \times \ldots \times \G_n}_{k\text{ times}}$ and $k\ge 0$. 
Define morphisms $d:C^k(\G_n,N) \to C^{k+1}(\G_n,N)$ by
\[
\begin{split}
(df)(g_1,\ldots,g_{k+1}) = g_1f(g_2,\ldots,g_{k+1}) &+ \sum_{j=1}^k (-1)^jf(g_1,\ldots,g_j g_{j+1},\ldots,g_{k+1})\\
& + (-1)^{k+1}f(g_1,\ldots,g_{k}). 
\end{split}
\]
One can check that $d^2=0$ and thus we obtain a complex $C^\ast(\G_n,N)$. We then define $H_c^*(\G_n,N)$ as $ H^*(C^\ast(\G_n,N),d)$. Of course, the same definition holds for any closed subgroup $G \subset \G_n$. We refer the reader to~\cite[Section 2]{davis06} for a more thorough discussion of various notions of continuous group cohomology used in chromatic homotopy theory. 
\begin{proof}
As previously we have that
\[
\begin{split}
\wExt_{\mE E}^{\ast,\ast}(E_*,\mE X) &\simeq H^*(\Hom_{\wComod_{\mE E}}(E_*,  (\mE E)^{\btimes (\ast+1)}\btimes_{E_*} \mE X)) \\
& \simeq H^*(\Hom_{E_*}(E_*,(\mE E)^{\btimes \ast}\btimes_{E_*} \mE X)) \\
& \simeq H^*( (\mE E)^{\btimes \ast} \btimes_{E_*} \mE X).
\end{split}
\]
By \Cref{lem:iteratedkunneth} we see there is an equivalence $\mE E^{\btimes \ast} \simeq \mE (E^{\wedge \ast})$, which is isomorphic to $\Hom^c(\G_n^{ \ast},E_*)$~\cite[p.~9]{devhop04}. This is pro-free by~\cite[Theorem 2.6]{hovoperations} and so applying~\Cref{lemma:lcompletetotensor} we see that 
\[
\begin{split}
(\mE E)^{\btimes \ast} \btimes_{E_*} \mE X &\simeq  \Hom^c(\G_n^\ast,E_*)\btimes_{E_*} \mE X \\
&\simeq \Hom^c(\G_n^\ast,E_*) \htimes_{E_*} \mE X.
\end{split}
\]
Since, under our assumptions, $\mE X$ is $\mathfrak{m}$-adically complete we have that 
\[
\Hom^c(\G_n^\ast,E_*) \htimes_{E_*} \mE X \simeq \Hom^c(\G_n^\ast,\mE X). 
\]
Then 
\[
H^*((\mE E)^{\btimes \ast} \btimes_{E_*} \mE X ) \simeq H^*(\Hom^c(\G_n^\ast,\mE X)).
\]

As in~\cite[Proof of Theorem 5.1]{lawsonnaumann} one can see that the latter is precisely $H_c^\ast(\G_n,\mE X)$. 
    
\end{proof}
\begin{rem}\label{rem:specseq}
In~\cite{davislawson} the authors give several examples where the $E_2$-term of the $K(n)$-local $E$-Adams spectral sequence for $\pi_*L_{K(n)}X$ can be identified with continuous group cohomology; we compare~\Cref{thm:COR} with these.  The following cases are considered.
\begin{enumerate}[(a)]
    \item By~\cite[Theorem 2(ii)]{devhop04}, if $X$ is finite then $E_2^{s,\ast} = H_c^s(\G_n,E_*X)$. If $X$ is finite, then smashing with it commutes with localisation and so $E_*X = \mE X$. By induction on the number of cells one can check that if $X$ is finite then $K_*X$ is finite (in each degree), where $K$ is the 2-periodic version of Morava $K$-theory used in~\Cref{rem:ktheory}.  By~\cite[Theorem 8.6]{hs99} this is equivalent to $\mE X$ being finitely generated.
    \item By~\cite[Theorem 5.1]{lawsonnaumann} if $E_*X$ is a flat $E_*$-module then $E_2^{s,\ast} = H^s_c(\G_n,\mE X)$. But by~\Cref{rem:tamemodules} if $E_*X$ is flat then $\mE X = L_0(E_*X) = (E_*X)^\wedge_\frak{m}$ is pro-free.
    \item By~\cite[Proposition 7.4]{hms} if $\mathcal{K}_{n,\ast}(X)$ is finitely generated as an $E_*$-module then $E_2^{s,\ast}=H^s_c(\G_n,\mathcal{K}_{n,\ast}(X))$. Here $\mathcal{K}_{n,\ast}(X) = \varprojlim_I E_*(X \wedge M_I)$. We suspect, but have been unable to prove, that if $\mathcal{K}_{n,\ast}(X)$ is finitely generated then $\mathcal{K}_{n,\ast}(X) \simeq \mE X$. We note that if $\mE X$ is finitely generated, then $X$ is dualisable~\cite[Theorem 8.6]{hs99}, and in this case the $\lim^1$ term in the Milnor exact sequence~\eqref{eq:milnor} vanishes~\cite[Proposition 6.2]{barthel2013completed}, so that $\mE X \simeq \mathcal{K}_{n,\ast}(X)$. 
    \item The last case considered is more complex. Let $X$ be a spectrum such that, for each $E(n)$-module spectrum $M$, there exists a $k$ with $\frak{m}^kM_*X =0$. Here $E(n)$ is the $n$-th Johnson-Wilson theory. Then, by~\cite[Proposition 6.7]{devhop04}, $E_2^{\ast,\ast} = H^\ast_c(\G_n,E_*X)$. Note that $E$ is an $E(n)$-module spectrum and so the proof of~\cite[Proposition 6.11]{devhop04} implies that $E \wedge X$ is $K(n)$-local, so that $E_*X = \mE X$. Since $E$ is an $E(n)$-module spectrum, $\mE X$ is a bounded $\frak{m}$-torsion module and so~\Cref{thm:COR} applies. 
\end{enumerate}
\end{rem}
\section{The category of Morava modules}
In this section we will show how~\Cref{cor:e2htfp} allows us to easily extend a result originally proved in~\cite{GHMR} for finite subgroups of $\G_n$ to arbitrary closed subgroups. First we need a definition.
\begin{definition}~\cite{GHMR}
    A \emph{Morava module} is a complete $E_*$-module $M$ equipped with a continuous $\G_n$-action such that, if $g \in \G_n,a \in E_*$ and $x\in M$, then 
    \[
g(ax) = g(a)g(x). 
    \]
    \end{definition}
    We denote the category of Morava modules by $\mathcal{EG}_n$. Here a homomorphism $\phi:M \to N$ of Morava modules is a continuous (with respect to the $\frak{m}$-adic topology) $E_*$-module homomorphism such that
    the following diagram commutes, where $g \in \G_n$:
  \[ \xymatrix{
  M \ar[r]^{\phi} \ar[d]_g & N \ar[d]^g\\
  M \ar[r]_{\phi} & N.
  }\]
   For example, if $X$ is any spectrum such that $\mE X$ is either finitely-generated, pro-free, or has bounded $\frak{m}$-torsion, then $\mE X$ is a complete $E_*$-module (by~\Cref{prop:completemodules}) and the $\G_n$-action on $E$ defines a compatible action on $\mE X$. This gives $\mE X$ the structure of a Morava module. The category $\mathcal{EG}_n$ is a symmetric monoidal category; given Morava modules $M$ and $N$ their monoidal product is given by $M \htimes_{E_*} N$ with the diagonal $\G_n$-action. 

   A homomorphism of complete $E_*$-modules is a homomorphism of $E_*$-modules that is continuous with respect to the $\frak{m}$-adic topology. However, it turns out that any homomorphism between complete $E_*$-modules is automatically continuous.  We learnt this from Charles Rezk, who also provided the following proof. 
\begin{lemma}\label{lem:cont}
Let $f: M \to N$ be an $E_*$-module homomorphism between complete $E_*$-modules. Then $f$ is continuous with respect to the $\frak{m}$-adic topology. 
\end{lemma}
\begin{proof}
    The map $f$ is an $E_*$-module homomorphism and so $f(\mathfrak{m}^k M)$ is a subset of $\mathfrak{m}^k N$ for any $k \ge 0$; thus $\mathfrak{m}^kM$ is a subset of $f^{-1}(\mathfrak{m}^kN)$. Therefore $f^{-1}(\mathfrak{m}^kN)$ is a union of $\mathfrak{m}^kM$-cosets. It follows from the fact that $\mathfrak{m}^kM$ is open that $f^{-1}(\mathfrak{m}^kN)$ is open in the $\mathfrak{m}$-adic topology.
\end{proof}

    Let $\{U_i\}$ be a system of open normal subgroups of $\G_n$ such that $\bigcap_i U_i = \{ e \}$ and $\G_n = \varprojlim_i \G_n/U_i$. Then we define $E_*[\![\G_n]\!] = \varprojlim_i E_*[\G_n/U_i]$, the completed group ring, with diagonal $\G_n$-action. If $H$ is a closed subgroup of $\G_n$, then we define $E_*[\![\G_n/H]\!]$ in a similar way, with diagonal $\G_n$-action.   With this in mind there is the following result. 
    \begin{prop}~\cite[Theorem 2.7]{GHMR}\label{prop:ghmr2.7}
     Let $H_1$ and $H_2$ be closed subgroups
of $\G_n$ and suppose that $H_2$ is finite. Then there is an isomorphism 
\[
E_\ast[\![\G_n/H_1]\!]^{H_2} \mathop{\longrightarrow}^{\cong} 
\Hom_{\mathcal{EG}_n}(\mE E^{hH_1},\mE E^{hH_2})
\]

    \end{prop}
We will in fact see that this holds more generally whenever $H_2$ is a closed subgroup of $\G_n$. 

We will need the following relationship between homomorphisms of Morava modules and $L$-complete comodules. 
\begin{prop}
	If $M$ and $N$ are both Morava modules and $L$-complete comodules, such that the underlying $L$-complete $E_*$-modules are pro-free, then 
	\[
\Hom_{\mathcal{EG}_n}(M,N) \simeq \Hom_{\wComod_{\mE E}}(M,N).
	\]
\end{prop}
\begin{proof}
	 Let $\phi:M \to N$ be a homomorphism of Morava modules. 	Note that $M$ and $N$ are complete, and hence also $L$-complete, and so $\phi$ defines a morphism in $\Mod_{E_*}$. We wish to show that this is in fact a comodule homomorphism. Let $\psi_M:M \to \Hom^c(\G_n,M) \simeq \mE E \htimes_{E_*} M$ be the adjoint of the $\G_n$-action map, and similarly for $\psi_N$. By~\Cref{lemma:lcompletetotensor} $\mE E \htimes_{E_*} M \simeq \mE E \btimes_{E_*} M$, and equivariance of $\phi$ implies that the following diagram commutes
	\[ \xymatrix{
	M \ar[r]^-{\phi} \ar[d]_{\psi_M} &  N\ar[d]^{\psi_N}\\
	\mE E \btimes_{E_*} M \ar[r]_-{\text{id} \btimes \phi} & \mE E \btimes_{E_*} N,
	}
	\]
	so that $\phi$ defines a morphism of comodules. 

Conversely, suppose that we are given an $L$-complete comodule homomorphism $\Phi: M \to N$. Since $M$ and $N$ are complete~\Cref{lem:cont} implies that $\Phi$ is a homomorphism of complete $E_*$-modules. Given the structure map $\psi_M$ we define a $\G_n$-action on $M$ using the retract diagram
\[
M \xr{\psi_M} \mE E \btimes_{E_*} M \simeq \Hom^c(\G_n,E_*) \htimes_{E_*} M \xr{\text{ev}(g) \htimes \text{id}} M,  
\]
where $\text{ev}(g):\Hom^c(\G_n,E_*) \to E_*$ is the evaluation map at $g \in \G_n$. The fact that $\Phi$ is a $L$-complete comodule homomorphism shows that, with this $\G_n$-action, $\Phi$ is in fact a morphism of Morava modules. 

These constructions define maps $\Hom_{\mathcal{EG}_n}(M,N) \to \Hom_{\wComod_{\mE E}}(M,N)$ and vice-versa, and it is not hard to see that these are inverse to each other. 
\end{proof}
\begin{cor}\label{cor:homequivalence}
	If $\mE X$ and $\mE Y$ are pro-free, then 
		\[
\Hom_{\mathcal{EG}_n}(\mE X,\mE Y) \simeq \Hom_{\wComod_{\mE E}}(\mE X,\mE Y).
	\]
\end{cor}
\begin{proof}
	The condition that $\mE X$ and $\mE Y$ are pro-free ensures that they are Morava modules; they are also $L$-complete $\mE E$-comodules by~\Cref{prop:completemodules}.
\end{proof}
We can now easily derive our version of the Goerss--Henn--Mahowald--Rezk result. 
\begin{prop}\label{prop:GHMR}
        Let $H_1$ and $H_2$ be closed subgroups
of $\G_n$. Then there is an isomorphism 
\[
E_\ast[\![\G_n/H_1]\!]^{H_2} \mathop{\longrightarrow}^{\cong} 
\Hom_{\mathcal{EG}_n}(\mE E^{hH_1},\mE E^{hH_2}).
\]
\end{prop}
\begin{proof}
By~\Cref{cor:e2htfp} we have 
\[
\wExt^{s,\ast}_{\mE E}(\mE E^{hH_1},\mE E^{hH_2}) \simeq H_c^s(H_2,E^{-\ast} E^{hH_1}).
\]
From the results of~\cite{hovoperations,devhop04} it can be deduced that $E^{-*}E^{hH_1} \simeq E_*[\![\G_n/H_1]\!]$ for any closed subgroup $H_1 \subset \G_n$, and that this isomorphism respects the $\G_n$-actions on both sides. Using~\Cref{lem:ext0} 
\[
\begin{split}
\Hom_{\wComod_{\mE E}}(\mE E^{hH_1},\mE E^{hH_2}) &\simeq H_c^0(H_2,E_{\ast}[\![\G_n/H_1]\!])\\& \simeq E_{*}[\![\G_n/H_1]\!]^{H_2}.
\end{split}
\]
Since $\mE E^{hH_1}$ and $\mE E^{hH_2}$ are pro-free,~\Cref{cor:homequivalence} implies that
\[
\begin{split}
   \Hom_{\wComod_{\mE E}}(\mE E^{hH_1},\mE E^{hH_2})  \simeq \Hom_{\mathcal{EG}_n} (\mE E^{hH_1},\mE E^{hH_2}),
\end{split}
\]
so that
\[
\Hom_{\mathcal{EG}_n} (\mE E^{hH_1},\mE E^{hH_2}) \simeq E_*[\![\G_n/H_1]\!]^{H_2}
\]
as required.
\end{proof}
\begin{rem}
    Let $H$ be a topological group and assume that $R$ is an $H$-spectrum and $X = \lim_i X_i$ is an inverse limit of a sequence of finite discrete $H$-sets $X_i$, such that $X$ has a continuous $H$-action. Following~\cite{GHMR} we define the $H$-spectrum
\[
R[\![X]\!] = \holim_i R \wedge (X_i)_+,
\]
with the diagonal $H$-action. In~\cite{behdav} Behrens and Davis show that if $H_1$ and $H_2$ are as above then there is an equivalence
    \[
F(E^{hH_1},E^{hH_2}) \simeq E[\![\G_n/H_1]\!]^{hH_2}.
    \]
This was originally proved for $H_2$ finite in~\cite{GHMR}. Combined with~\Cref{prop:GHMR} it is easy to see that there is a commutative diagram
\[
\begin{tikzcd}
    \pi_* E[\![\G_n/H_1]\!]^{hH_2} \arrow{r} \arrow[swap]{d}{\simeq} & \left( E_*[\![\G_n/H_1]\!]\right)^{H_2} \arrow{d}{\simeq} \\
    \pi_*F(E^{hH_1},E^{hH_2}) \arrow{r} & \Hom_{\mathcal{EG}_n}(\mE E^{hH_1},\mE E^{hH_2}),
\end{tikzcd}
\] 
for $H_1,H_2$ closed subgroups of $\G_n$. Again this is proved in~\cite{GHMR} under the additional assumption that $H_2$ is finite.
\end{rem}

\section{The $E_1$ and $K(1)$-local $E_1$-Adams spectral sequences}
Since $E_*E$ is a flat $E_*$-module we have the $E$-Adams spectral sequence (see, for example~\cite{hovsad})
\[
E_2^{s,t} = \Ext^{s,t}_{E_*E}(E_*,E_*X) \Rightarrow \pi_*L_nX.
\]
This is a spectral sequence of $\Zp$-modules and our goal in this section is to use the derived functor of $p$-completion on $\Zp$-modules to construct a spectral sequence abutting to the $E_2$-term of the $K(n)$-local $E$-Adams spectral sequence. 

Unfortunately our proof only works when $n=1$ and $p$ is an arbitrary prime.  We shall see that, for a spectrum $X$, the spectral sequence naturally carries copies of $\Q/\Z_{(p)}$ in $\Ext^{\ast,\ast}_{E_*E}(E_*,E_*X)$ to copies of $\Z_p$ in $H^{\ast}_c(\G_n,\mE X)$. Already at height 2, for primes greater than or equal to 5, the calculations of~\cite{shimyabe} imply that there are 3 copies of $\Q/\Z_{(p)}$ which lie in bidegree $(4,0),(4,0)$ and $(5,0)$, whilst in $H^*_c(\G_2,(E_2)_*)$ there are copies of $\Z_p$ in bidgrees $(0,0),(1,0)$ and $(3,0)$. The grading on the spectral sequence we construct will imply that there is no possible class that could give rise to the copy of $\Z_p$ in bidegree $(1,0)$. If one accepts the chromatic splitting conjecture~\cite{hovbous} then an analogue of our spectral sequence cannot exist at all when $n \ge 2$.

\textbf{Note:} From now on, unless otherwise stated, it is implicit that $E$ refers to Morava $E$-theory at height 1 only. 

The reason that the spectral sequence exists when $n=1$ is due to the fact that $E_* \simeq \Z_p[u^{\pm 1}]$. As Hovey shows in~\cite[Lemma 3.1]{hovey08fil}, given a graded $E_*$-module $M$, there is an isomorphism $(L_0M)_k \simeq L_0M_k$ where the second $L_0$ is taken in the category of $\Z_p$-modules. A similar result holds for completion with respect to $\Zp$-modules. We will often use this implicitly to pass between ungraded $\Zp$-completion and graded $E_*$-completion. 
\begin{theorem}\label{thm:specseq}

          If $E_*X$ is a projective $E_*$-module and $E_*Y$ is a flat $E_*$-module, then there is a spectral sequence
       \[
E_2^{i,s} = L_i \Ext_{E_*E}^{s,t}(E_*X,E_*Y) \Rightarrow \wExt_{\mE E}^{s-i,t}(\mE X,\mE Y),
       \]
       where $L_i$ is taken in the category of $\Zp$-modules. 
\end{theorem}
\begin{rem}
  The assumption that $E_*X$ is projective ensures that $\Ext_{E_*E}^{\ast,\ast}(E_*X,E_*Y)$ can be computed by a relative injective resolution of $E_*Y$. Additionally, if $E_*X$ is projective, the spectral sequence of \cite[Theorem 2.3]{hovey08fil} collapses, so $\mE X \cong  L_0E_*X$ is pro-free by~\cite[Proposition A.15]{barthel2013completed}. 
\end{rem}
\begin{proof}
   Let $M^\ast$ be the cobar complex with $M^s = E_*E^{\otimes (s+1)} \otimes_{E_*} E_*Y$. Then 
   \[
   \begin{split}
 \Ext_{E_*E}^{\ast,\ast}(E_*X,E_*Y) &= H^*(\Hom_{E_*E}(E_*X,E_*E^{\otimes (\ast+1)} \otimes_{E_*} E_*Y))    \\
 &\simeq H^*(\Hom_{E_*}(E_*X,E_*E^{\otimes \ast} \otimes_{E_*} E_*Y))
    \end{split}
   \]
For brevity, we denote $\Hom_{E_*}(E_*X,E_*E^{\otimes \ast} \otimes_{E_*} E_*Y)$ by $N^\ast$. Since $E_*E$ and $E_*Y$ are flat $E_*$-modules, so is the iterated tensor product $E_*E^{\otimes \ast} \otimes_{E_*} E_*Y$; since $E_*X$ is projective, and $E_*$ is Noetherian, each $N^s$ is flat, and hence tame. Under such a condition there is a spectral sequence\footnote{Note that we switch from a chain complex to a cochain complex, which accounts for the shift in grading in the abutment.}~\cite[Proposition 8.7]{rezk2013} 
   \[
E_2^{i,s} = L_i(H^{s}(N^\ast)) \simeq L_i\Ext_{E_*E}^{s,\ast}(E_*X,E_*Y) \Rightarrow H^{s-i}(L_0(N^\ast)).
   \]

To identify the abutment use the spectral sequence of~\cite[Theorem 2.3]{hovey08fil} to see that $L_0(E_*X) = \mE X$ and $L_0(E_*Y) = \mE Y$, since both $E_*X$ and $E_*Y$ are flat. It also follows from the symmetric monoidal structure on $L$-complete $E_*$-modules~\cite[Corollary A.7]{hs99} that $L_0(E_*E^{\otimes \ast} \otimes_{E_*} E_*Y) = L_0(E_*E)^{\btimes \ast} \btimes_{E_*} L_0(E_*Y)$. Applying~\Cref{cor:homlcompletion} we see that 
\[
\begin{split}
L_0(N^\ast) &= \Hom_{\Mod_{E_*}}(\mE X,\mE E^{\btimes \ast} \btimes_{E_*} \mE Y) \\
  &\simeq \Hom_{\wComod_{\mE E}}(\mE X,\mE E^{\btimes (\ast+1)} \btimes_{E_*} \mE Y).
\end{split}
\]
The cohomology of the latter is precisely $\wExt^{\ast,\ast}_{\mE E}(\mE X,\mE Y)$. 
\end{proof}
\begin{rem}
  This spectral sequence can also be obtained as a Grothendieck spectral sequence. Again assuming that $E_*X$ is projective, consider the following functors, and their derived functors:
    \[
G: \Hom_{E_*E}(E_*X,-), \quad R^tG: \Ext^t_{E_*E}(E_*X,-) 
\]
from $E_*E$-comodules with flat underlying $E_*$-module to $\Zp$-modules, and
\[
F: L_0(-), \quad L^tF : L_t(-),                                         
\]
from $\Zp$-modules to $\Z_p$-modules. Then
\[
FG(-) = L_0\Hom_{E_*E}(E_*X,-).
\]
Let $E_*E \otimes_{E_*} N$ be an extended $E_*E$-comodule, where $N$ is a flat $E_*$-module; this implies $E_*E \otimes_{E_*} N$ is still flat. Then, for $s>0$,
\[
\begin{split}
  L^sF(G(E_*E \otimes_{E_*} N)) & \simeq L_s(\Hom_{E_*E}(E_*X,E_*E \otimes_{E_*} N)) \\
  & \simeq L_s \Hom_{E_*}(E_*X,N) = 0. \\
\end{split}
\]
by~\Cref{cor:homlcompletion}. This implies that the Grothendieck spectral sequence exists. To identify the abutment we just need to identify the derived functors of $FG$. Once again we can use the cobar resolution $M \to E_*E \otimes_{E_*} M \to \cdots $, where $M$ is an $E_*E$-comodule that is flat as an $E_*$-module. Then 
\[
\begin{split}
R^sFG(M) &= H^s(L_0\Hom_{E_*E}(E_*X,(E_*E)^{\otimes (\ast+1)}\otimes_{E_*}M)) \\
  &\simeq H^s(L_0\Hom_{E_*}(E_*X,(E_*E)^{\otimes \ast}\otimes_{E_*}M)) \\
  & \simeq H^s(\Hom_{\Mod_{E_*}}(\mE X,\mE E^{\btimes \ast} \btimes_{E_*} L_0M )) \\
  & \simeq H^s(\Hom_{\wComod_{\mE E}}(\mE X,\mE E^{\btimes (\ast+1)}\btimes_{E_*} L_0M) ).
\end{split}
\]
As we have seen previously this identifies $R^sFG(M)$ with $\wExt^{s,\ast}_{\mE E}(\mE X,L_0M)$.
\end{rem}
\subsection{Height 1 calculations}
As an example we will show how the calculation of $H^*_c(\G_1,E_*)$ follows from the corresponding calculation of $\Ext^{\ast,\ast}_{E_*E}(E_*,E_*)$. We work at the prime 2 since the calculations are more interesting here, due to the presence of 2-torsion in $\G_1 = \Z_2^\times \simeq \Z_2 \times \Z/2$. We first need the following lemma that relates the $E(n)$ and $E$-Adams spectral sequences. This result holds for \emph{all} heights $n$ and primes $p$. 

\begin{lemma}[Hovey-Strickland]\label{prop:Enss}
    Let $M$ and $N$ be $E(n)_*E(n)$-comodules. Then, for all $s$ and $t$, there is an isomorphism
    \[
\Ext^{s,t}_{E(n)_*E(n)}(M,N) \simeq \Ext^{s,t}_{E_*E}(M \otimes_{E(n)_*} E_{*},N \otimes_{E(n)_*} E_{*}).
    \]
\end{lemma}
\begin{proof}
    By~\cite[Theorem C]{hoveystricklandcomod} the functor that takes $M$ to $M \otimes_{E(n)_*}E_*$ defines an equivalence of categories between $E(n)_*E(n)$-comodules and $E_*E$-comodules.
\end{proof}
This implies that $\Ext_{E(1)_*E(1)}^{s,t}(E(1)_*,E(1)_*) =  \Ext_{E_*E}^{s,t}(E_{\ast},E_{\ast})$ for all $s$ and $t$. We start with a calculation described in~\cite[Section 6]{hovsad}.\footnote{Note that there is a small typo in~\cite{hovsad}. Namely the $\Z/16k$ referred to by Hovey-Sadofksy in filtration degree 1 of the $8k-1$-stem should actually refer to the 2-primary part of $\Z/16k$.}
\begin{prop}
    Let $p = 2$. Then
    \[
\Ext_{E(1)_*E(1)}^{s,t}(E(1)_*,E(1)_*)= \begin{cases}
    \Z_{(2)} & t=0,s=0 \\
    \Q/\Z_{(2)} & t=0,s=2\\
    \Z/2^{k+2} & t=2^{k+1}m,m \not \equiv 0 \mod (2),k \ne 0, s=1 \\
    \Z/2 & t = 4t'+2, s=1, t' \in \Z \\
    \Z/2 & s \ge 2, t=\text{even} \\
    0 & \text{else.}
\end{cases}
    \]
\end{prop}
We now run the spectral sequence of~\Cref{thm:specseq}. Note that~\Cref{examp:zp} computes the $E_2$-term of this spectral sequence. It can be checked that the differentials are $d_r:E_r^{i,s} \to E_r^{i+r,s+r-1}$; since the spectral sequence is non-zero only for $i=0$ and $i=1$ we see that there are no differentials in the spectral sequence, and that it collapses at the $E_2$-page. We deduce the following:
   \begin{theorem}
      Let $p = 2$. Then
    \[
H^s_c(\G_1,E_t) = \begin{cases}
    \Z_2 & t=0,s=0,1 \\
    \Z/2^{k+2} & t=2^{k+1}m,m \not \equiv 0 \mod (2),k \ne 0, s=1 \\
    \Z/2 & t = 4t'+2, s=1,t' \in \Z \\
    \Z/2 & s \ge 2, t=\text{even} \\
    0 & \text{else.}
\end{cases}
    \]
\end{theorem}

\bibliographystyle{amsalpha}
	\bibliography{BibtexDatabase}
\end{document}